\documentclass[12pt,a4paper]{article}
 \usepackage{amsmath,amstext,amssymb,amscd}
 \usepackage{graphicx}

\newtheorem{theorem}{Theorem}[section]
\newtheorem{lemma}[theorem]{Lemma}
\newtheorem{corollary}[theorem]{Corollary}

\newtheorem{prop}[theorem]{Proposition}

\newenvironment{proof}{\noindent{\bf Proof}~}%
{\hfill$\qed$\medskip}

\def\qed{\Box}

\begin{document}

\parskip 2pt

\title{Stable and meta-stable contract networks%
 \thanks{Supported by grant RFFI 20-010-00569-A}}

\author{Vladimir I. Danilov\thanks{Central Institute of Economics and
Mathematics of the RAS, Nahimovskii prospect, 47, 117418 Moscow; email: vdanilov43@mail.ru.}
 \and
Alexander V. Karzanov\thanks{Central Institute of Economics and
Mathematics of the RAS, Nahimovskii prospect, 47, 117418 Moscow; email: akarzanov7@gmail.com.}}

\date{}
\maketitle

\begin{abstract}
We consider a hypergraph $(I,C)$, with possible multiple (hyper)edges and loops, in which the vertices $i\in I$ are interpreted as \emph{agents}, and the edges $c\in C$ as \emph{contracts} that can be concluded between agents. The preferences of each agent $i$ concerning the contracts where $i$ takes part are given by a \emph{choice function} $f_i$ possessing the so-called \emph{path independent} property. In this general setup we introduce the notion of stable contract network.

The paper contains two main results. The first one is that a general stable contract problem for $(I,C,f)$ is reduced to a special one in which preferences of the agents are given by \emph{weak orders}, or, equivalently, utility functions. However, stable contract systems may not exist. Trying to overcome this trouble, we introduce a weaker notion of \emph{meta-stability} for systems of contracts. Our second result is that meta-stable systems always exist. A proof of this result relies on an appealing theorem on the existence of the so-called compromise function.


 \end{abstract}

      \section{Introduction}

In their lives, people often have to make joint actions and organize groups in order to achieve some goals. We call such cooperations \emph{contracts}. Examples are:  house exchange, purchase or sale, marriage, loan or deposit of money, hiring, co-financing society, cartel, military or economic union of countries. Contracts can involve not only individuals, but also larger entities; for convenience, we call the parties of a contract \emph{agents} or \emph{participants}. Some contracts include only two agents (we call such contracts \emph{binary}), but many other ones can  include a larger number of agents. Note also that agents are allowed to enter several different contracts at once.

Contracts bring some benefits to the participants, but also may require them to spend money, time or other expenses. It is important for participants to know more precisely what they can count on. For this purpose, the agreements should be as detailed and formalized as possible, though not everything can be taken into account. For example, a marriage contract may include how much time the spouses can spend in the family and how much `outside', how to share household efforts, how many children to have, etc. The more all this is worked out in detail, the better the participants represent the pros and cons and can compare different contracts. We further assume that each contract can be unambiguously evaluated by each of its participants.

The information about the `preference' of one or another contract primarily affects the choice of  contracts that will actually be concluded (signed). Here we are based on the premise of voluntariness of contracts. No one can force an agent to sign a contract if the agent does not like to do this. On the other hand, no one can forbid a group to sign a contract if its all members agree. These two requirements lead to the concept of a \emph{stable} system (or network) of contracts, which will be the main topic of our work. The concept of stability originally appeared in the work of Gale and Shapley~\cite{GS} and gradually has become the subject of extensive researches both theoretical and practical. Gale and Shapley showed that in case of marriages, a stable system always exists. They assumed the marriages to be  monogamous and bipartite (heterosexual). Without these conditions, stable systems may not exist. An example is the famous problem of  stable `roommates', or `division into pairs'. Even to a greater extent, this trouble concerns non-binary contracts.

One remark is needed to be mentioned here. When agents are allowed to conclude many contracts, they should be able to compare not only individual contracts, but also arbitrary subsets of contracts. Therefore, it is not enough to attribute a value to each contract only. Instead, we prefer to use the so-called \emph{choice functions} (CFs, for brevity), which tell us what groups of contracts from the available list are `the best' ones to be chosen for signing. This approach was initiated by Kelso and Crawford~\cite{KC} who revealed importance of the condition of "substitutability" for the existence of stable systems. Subsequently, a number of researches have shown that this condition is applicable to all problems with `bilateral' contracts (see e.g.~\cite{F,HM}). We show that in a more general setting, this condition on agents' CFs is also adequate.

The paper contains two main results. The first one is that a general problem on stable contract systems is reduced to a more special situation in which preferences of agents are described by use of \emph{weak orders}, or utility functions. Roughly speaking, agents conclude contracts having the maximum utility and ignore the rest. However, even in such situations the stability need not exist. To overcome this trouble, we propose a weaker notion of \emph{meta-stable} contract system.

{Roughly speaking, the notion of meta-stability differs from that of stability by withdrawing one of the two defining axioms for the latter (namely, the one of individual rationality); in the end of Section~5 we briefly discuss justifications for the release of this sort. An important fact (established in Proposition~4.1) is that every stable contract system (if exists) turns out to be meta-stable as well. Our second main result in this paper is that meta-stable systems always exist. These appealing properties give a good theoretical ground for studying meta-stable systems, which, to our belief, would find interesting economic applications.}


\section{Basic definitions and settings}

A general setup can be stated as follows. There are a finite set of \emph{agents} $I$ and a finite set $C$ of \emph{contracts} which are available for the agents to conclude. Each agent can conclude several contracts. Each contract $c\in C$ is shared by a nonempty set of participants $P(c)\subseteq I$. If $P(c)$ is a singleton $\{i\}$, the contract $c$ is called \emph{autarkic}; this can be thought of not as a contract in reality, but rather as an `activity' available to $i$ alone. Thus, the object that we deal with can be described as a hypergraph, with possible parallel  hyperedges, and when needed, we may use the language of (hyper)graphs, interpreting the vertices as agents and  (hyper)edges as contracts. Equivalently, the input can be encoded as a bipartite graph with the parties (color classes) $I$ and $C$.

For $S\subseteq C$, let $S(i)$ denote the set of contracts $s\in S$ such that $i$ is a  participant of $s$.

As mentioned above, the quality of contracts can be expressed in terms of  'utility', which brings benefits to their participants. Guided by these utilities, agents conclude some contracts and refuse other ones. So an evaluation of contracts via utilities is an important part of the problem. The simplest way to define a utility is to express it numerically, by assigning a (real or integer) number $u_i(c)$ to each contract $c$ of $C(i)$. However, this is not the most general way to establish `preferences' of agents. Since an agent can conclude several contracts, it is often important for him/her to know the utilities not only for individual contracts, but also for their collections.

A rather powerful method to describe  preferences of contracts, yielding sufficient flexibility and generality, attracts choice functions. A {\em choice function} $f$ on an (abstract) set of `alternatives' $X$ selects a `good' subset $f(A)\subseteq A$ for any set (`menu') $A\subseteq X$. In our case, the choice of agent $i$ is taken within the set  of available contracts $C(i)$.\footnote{ Here we default assume that an agent does not care of what contracts are concluded without his participation. For example, in the situation of hiring workers by firms, one assumes that it is important to the worker in which firms he will work, but it does not matter to him who else works in these firms. On the other hand, it is important to a firm who will work in it, but it does not matter where else the employee works. In some situations, such an assumption looks not realistic, but it can be accepted as a first approximation.}

It light of this, the second important ingredient of the problem consists of an appropriate set of choice functions $f_i$ on the sets $C(i)$ for agents $i\in I$. We refer to such a set of choice functions representing preferences of agents as an \emph{equipment} of the hypergraph $G=(I,C)$. Using this, we now can talk about the stability of a contract system (or network) $S\subseteq C$. Roughly speaking, this is a system $S$ such that nobody wants to change it, either by renouncing some contract, or by concluding a new contract, perhaps by breaking some existing ones in $S$. At the same time, it is assumed that contracts are concluded voluntarily. This means that any agent can refuse to conclude any contract, and that any contract can be concluded only with the consent of its all participants. The formal definition is as follows.\medskip

\noindent\textbf{Definition.} A network $S$ is called {\em stable} if the following two conditions hold:

 \begin{itemize}
\item[\textbf{S0.}]    $f_i(S(i))=S(i)$ for any $i\in I$;
\item[\textbf{S*.}] If a contract $b$ does not belong to $S$, then $b\notin f_i(S(i)\cup b)$ for some $i\in P(b)$.
 \end{itemize}
 (Hereinafter, for a set $S$ and a singleton $s$, we may write $S\cup s$ for $S\cup\{s\}$.)
 \smallskip

The first condition expresses the possibility of renouncing any contract. The second one says that if a contract $b$ is interesting to its all participants, then it should be concluded. And the absence of $b$ in $S$ indicates incompleteness of the process of building a system of contracts. Sometimes one says that such a contract $b$ \emph{blocks} the system $S$.

The main issue that we will be dealing with concerns the existence of stable networks. The answer depends on both the structure (`geometry') of the original network $C$ and (to a greater extent) the `preferences' of agents. For example, if $(I,C)$ is a bipartite graph, then a stable network $S\subseteq C$ exists under rather weak conditions on preferences. And if all agents behave indifferently, then already the original network $C$ is stable. On the other hand, even in case of binary contracts (in a non-bipartite graph) with the best individual preferences, the stability may not take place.

We will focus on preferences without imposing preliminary restrictions on the original network $C$. As is mentioned earlier, preferences are given via choice functions. Note that a variety of possible CFs is large, but a majority of them do not correspond to an intuitive concept of  'reasonable choice'. To illustrate this situation, we demonstrate two examples of CFs that are regarded as `rational'.\medskip

\noindent\textbf{Example 1.} Let $\le$ be a preorder on a set $X$ (that is, a reflexive and transitive binary relation, admitting `equal' elements). And let $f(A)$ consist of all maximal elements in $A\subseteq X$ relative to  $\le$ (one often $f$ is denoted as $\max_\le$). A CF of this kind is considered as rational, since a rational reason for including one or another alternative in the choice is clearly seen,  to be the lack of a better alternative. Note that this choice is nonempty (when a menu $A$ is such).
\smallskip

Two special cases of this construction deserve to be mentioned. The first one is when $\le$ forms a {\em weak order} (a preorder in which any two elements are comparable). The second one is when $\le$ is a \emph{linear} (total strict) order, or ranking. In the latter case, the choice $f(A)$ consists of a single element (if $A\ne\emptyset$); in this case we are talking about \emph{linear preferences}, or linear equipment.\medskip

\noindent\textbf{Example 2.} Let $\le$ be a linear order, but the choice includes $b$ best items from menu. The number $b$ is prescribed and can be understood as a quota. Such a choice rule is viewed as rational as well; it has been considered in many works on stable $b$-matchings. See e.g.~\cite{CF,F,F-10,IS} where generalizations (of type `many-to-many') of stable marriages and roommates are studied.\medskip

Both examples are special cases of the so-called \emph{path-independent} choice functions. Such a function $f$ satisfies the following functional equation:
$$
                                f(A\cup B)=f(f(A)\cup B).
                     $$

This equation says that the answer does not change under replacing a part of the menu by its best elements. Therefore, the computation of $f(\cdot)$ can be performed step by step, by choosing every time an arbitrary subset in the current domain, and the answer does not depend on the choice (or `path'). This condition was introduced by Plott~\cite{P}, and we call such CFs \emph{Plottian} or \emph{Plott functions}. It turned out that this condition is quite suitable for studying stability, and  later on we will assume throughout that all CFs in question are just Plottian. Such CFs  have been investigated extensively in the literature (especially Aizerman and Malishevski's paper~\cite{AM} should be distinguished); the facts about Plottian CFs that are needed for us are contained in Appendix~C.

Note that when the equipment is given by Plott functions, any stable network (if exists) is Pareto optimal. However, one can address the question: how to understand the optimality if preferences are given by CFs? We can do this in the following way. Let $f$ be a Plott function on a set $X$. One can associate with $f$ the following hyper-relation $\preceq =\preceq _f$ (a relation on $2^X$ introduced by Blair \cite{Bl}), which is given by the expression:
$$
                               A \preceq  B\quad \mbox{if}\quad f(A\cup B)\subseteq B.
                             $$
This hyper-relation is transitive and reflexive (when $f$ is Plottian). If CF $f$ is given by a weak order $\le$, then $A\preceq B$ implies $\max(A)\le \max(B)$. The following fact is of use.

\begin{prop} \label{prop0}
Suppose that all $f_i$'s are Plott functions  and that $\preceq _i$ are the hyper-relations as above. Let $S$ be a stable system, and let $T$ be a contract system satisfying condition \emph{\textbf{S0}}. If  $S(i)\preceq _i T(i)$ holds for each $i$, then $S=T$.
 \end{prop}

\noindent In other words, if the system $T$ is not worse than $S$ for all agents, then $S$ and $T$ coincide (in particular, $T$ is not better than $S$).
\medskip

 \begin{proof}
The condition $S(i)\preceq _i T(i)$ means that $f_i(S(i)\cup      T(i))\subseteq T(i)$. Since
  $$
f_i(S(i)\cup T(i))=f_i(S(i)\cup T(i)\cup T(i))=f_i(f_i(S(i)\cup T(i))\cup T(i))=f_i(T(i))
   $$
(in view of $f_i(S(i)\cup T(i))\subseteq T(i)$), we have $f_i(S(i)\cup T(i))=T(i)$.

If the opposite inequality $T(i)\preceq _i S(i)$ holds for all $i$, then $f_i(S(i)\cup T(i))=f_i(S(i))=S(i)$, then $S(i)=T(i)$, and we are done. So we may assume that for some agent (denote it as 0)  the set  $T(0)=f_0(S(0)\cup T(0))$ is not contained in $S(0)$. Then there is a contract $t$ belonging to $T(0)-S(0)$. Let $j$ be an arbitrary participant of the contract $t$; so $t\in T(j)=f_j(S(j)\cup T(j))$. Since $t$ is chosen (by the CF $f_j$) in the larger set $S(j)\cup T(j)$, this $t$ is also chosen in the smaller set $S(j)\cup t$ (by Heredity property, see Appendix~C). So we have  $t\in  f_j(S(j)\cup t)$, which is true for any element $j$ of $P(t)$. Since $t$ does not belong to $S$, we obtain a contradiction to the stability condition {\bf S*}.
 \end{proof}

The first main result of this paper is that the problem with general Plott functions can be reduced to a problem in which the preferences of agents are given by weak orders, as we explain in the next section.
\medskip

\textbf{Remark.} {In the above definition of a stable contract network $S$, condition {\bf S*} requires that no (individual) contract $b$ in $C-S$ can block $S$. This matches usual non-blocking conditions in the classical works on stability. In some last papers, a somewhat stronger condition of stability were  proposed, which forbids any blocking contract sets; see e.g.~\cite{RY, HK}. However, in the case of path-independent choice functions, both conditions become equivalent; see~\cite{HK}.}


\section{The reduction theorem}

\begin{theorem} \label{theorem1}
Suppose that the equipment of a hypergraph $(I,C)$ is given by Plott functions. Then there exist a hypergraph $(I',C')$ equipped with weak orders and a mapping (hypergraph homomorphism) $\pi : (I',C')\to (I,C)$ such that:

a)  for any stable system $S'$ in $C'$, its image $\pi(S')$ is stable in $C$;

b) for any stable system $S\subseteq C$, there is a stable system $S'$ in $C'$ such that $\pi(S')=S$.
 \end{theorem}

This assertion is based on a theorem in~\cite{AM} saying that any Plott function is representable as the union of several linear CFs. We construct the desired hypergraph $(I',C')$ by splitting each agent $i\in I$ into a set of its `subagents' $\widetilde i_j$, which already have weak orders as preferences on the contracts available to them. To simplify our description (and make the construction more transparent), in suffices to describe in detail only one step, consisting of a `splitting' operation for one agent, which is denoted as 0.

Let us assume that the CF $f_0$ of this agent is represented as the union of several `simpler'  Plott functions $f_1,\ldots,f_\ell$ (for example, linearly or weakly ordered ones).\footnote{Recall that the union $f_1\cup \cdots \cup f_\ell$ is given by the formula $(f_1\cup \cdots \cup f_\ell)(A)=f_1(A)\cup \cdots \cup f_\ell(A)$.} We even may assume that $\ell=2$, and accordingly `split' the agent 0 into two new agents $0_1$
and $0_2$, which are denoted simply as 1 and 2. Each of them has the same set $C(0)$ of contracts, but their preferences differ and are given by CFs $f_1$ and $f_2$, respectively. More formally, the new set of contracts $\widetilde C$ is arranged as follows:
                                              $$
      \widetilde      C:=(C-C(0))\sqcup  \widetilde C(1)\sqcup \widetilde  C(2),
                                                $$
where $\widetilde C(1)=C(0)\times\{1\}$, $\widetilde C(2)=C(0)\times\{2\}$ are two copies of $C(0)$.  When this is not confusing, we will identify each of $\widetilde C(1)$ and $\widetilde C(2)$ with $C(0)$. In other words, each contract $c$ involving agent 0 is duplicated, turning into two contracts, $c_1$ and $c_2$, concluded by the same agents except for 0 which is now replaced by 1 and 2, respectively. The mapping $\pi$ sends agents 1 and 2 to 0, and sends $c_1$ and $c_2$ to $c$.

We have already described the preferences of agents 1 and 2; namely, they are given by CFs $f_1$ and $f_2$. For the other agents (which will be usually denoted as $j$, $j'$, etc.), the contracts $c_1$ and $c_2$ are equivalent (they perceive them as contracts with agent 0). More formally, agent $j$ (considered as an element of the set $\widetilde I=(I-\{0\})\cup \{1,2\}$) chooses $\widetilde c$ from a menu $A\subseteq\widetilde C(j)$ if and only if $c=\pi(\widetilde c)$  is chosen from $\pi(A)$:
                                                      $$
      \widetilde      f_j(A)=A\cap \pi ^{-1}(f_j(\pi A)).
                                                        $$
In Appendix~C, we show that $\widetilde f_j$ is a Plott CF as well.

Note that even if the old CF $f_j$ were linear, the new CF $\widetilde f_j$ should be given by a weak order  in general, because the twins $c_1$ and $c_2$ for the agent $j$ are equivalent (indifferent). This is the reason why we are able to reduce the problem not to the linear case, but merely to the weakly ordered one.

So, we have described the new system $\widetilde C$ of contracts, and now we can formulate the first assertion; it will be proved in Appendix~A.

 \begin{prop} \label{prop1}
For $C$ and $\widetilde C$ as above, if $\widetilde S$ is a stable system in $\widetilde C$, then the system $S=\pi(\widetilde S)$ is stable in $C$.
 \end{prop}

Let \textbf{St}($G$) denote the set of stable networks for the equipped  hypergraph $G$, and \textbf{St}($\widetilde G$) a similar set for $\widetilde G$. Then Proposition~\ref{prop1} determines the mapping
                       $$
      \pi :   \textbf{St}(\widetilde G) \to \textbf{St}(G)
                 $$
that translates any stable system $\widetilde S$ into $\pi (\widetilde S)$. We claim that this mapping is surjective, that is, for any stable system $S$ in $G$, there exists a stable system $\widetilde S$ in $\widetilde G$ such that $\pi (\widetilde S)=S$. Moreover, we will build such a system $\widetilde S$ canonically.\medskip

\emph{Construction of} $\widetilde S$. If a contract $s$ belongs to $S$ and does not contain agent 0, then $s$ is lifted in $\widetilde C$ in a natural way and included in $\widetilde S$. Therefore, we only have to explain how to form $\widetilde S(1)$ and $\widetilde S(2)$. We put $\widetilde S(1):=f_1(S(0))$ and $\widetilde S(2):=f_2(S(0))$ (identifying $\widetilde C(1)$ and $\widetilde C(2)$ with $C(0)$).

\begin{prop} \label{prop2}
The system $\widetilde S$ constructed as above is stable.
 \end{prop}

 The proof is given in Appendix~A. \medskip

Now let us return to Theorem~\ref{theorem1}. A required mapping $\pi :G'\to G$ is constructed by iterating the above splitting construction. The induction is proceeded by the number of vertices $j$ in (the current hypergraph) $G$ for which CF $f_j$ is not weakly ordered. If there are no such vertices, we are done. So let 0 be a vertex of $G$ for which CF $f_0$ is not weakly ordered. Discarding unnecessary edges, we may assume that $f_0$ is not empty-valued. By Aizerman--Malishevski's theorem (see Appendix~C), $f_0$ is representable as $f_0=f_1\cup \cdots\cup f_\ell$ with all $f_1,\ldots,f_\ell$ implemented by  weak orders (and even by linear ones). Let $\widetilde G$ be the hypergraph constructed as above, but with splitting 0 not into two vertices, but into $\ell$ ones. The proofs given above can be extended in a natural way to this case as well. The CFs in the new vertices $1,\ldots,\ell$ are already weakly ordered. But what about the other vertices $j$? When $f_j$ was weakly ordered, the new $\widetilde f_j$ is again weakly ordered. Indeed (see Appendix~C), if $f_j$ was generated by a weak order $\le _j$ on $C(j)$, then $\widetilde f_j$ is generated by the weak order $\pi^*(\le _j)$ on $\widetilde{C}(j)$, where $\pi$ is the projection of $\widetilde C(j)$ on $C(j)$.  \hfill$\Box$\medskip

Thus, we obtain a reduction of a general case (with Plottian CFs) to the special case where the preferences of agents are given by weak orders. And now it is reasonable to analyze this special case.


\section{Meta-stable contract networks}

Stable networks need not exist even if the agent preferences are given by linear orders. This depends, to a large extent, on the structure of the hypergraph $(I,C)$. In the literature there are many papers containing results in this field, among those we can mention~\cite{AF, F, I, T}. Below we propose a new concept of meta-stable network, which, on the one hand, is close to the concept of stable network and, on the other hand, such networks `always' exist.

As before, we assume that the preferences of agents are given by non-empty-valued Plott functions $f_i$ on $C(i)$ for all $i\in I$. To introduce the meta-stability, we first need to specify the notion of domination.

Let $f$ be a non-empty-valued Plott function on an abstract set $X$. Let us say that an element $d\in X$ \emph{dominates} (in spirit of Scarf) a set $A\subseteq X$ if either $A$ is empty or there is $a\in A$ such that $a\notin f(\{a,d\})$. Intuitively, the dominance of an element $d$ means that $d$ is somehow `better' than the set $A$. In the case $A=\emptyset$, this means that any contract is better than nothing. In the case $A\ne\emptyset$, this means that there is at least one element of $A$ that is `worse' than $d$.  Roughly speaking, we estimate a subset $A$ according to the worst case scenario, that is, we focus on guaranteed results.\medskip

\noindent\textbf{Definition.} A contract set  $S\subseteq C$ is \emph{dominated} by a contract $d$ if for each participant $i$ of  $d$, this contract dominates (relative to $f_i$) the set  $S(i)$.\medskip

The presence of a dominating contract $d$ indicates some instability of the system $S$. All participants of $d$ are tempted to add $d$ to  $S$ (if $d\notin S$) and refuse the contracts worse than $d$. On the other hand, if $d$ does not dominate $S$, then at least one of its potential participants is not interested in concluding $d$.\medskip

\noindent\textbf{Definition.} A network of contracts $S$ is called \emph{meta-stable} if there is no contract dominating $S$.\medskip

\noindent\textbf{Remark.} Let $S$ be meta-stable, and a participant $i$ has an autarkic contract $a_i$. Then the set $S(i)$ is nonempty. Since if $S(i)=\emptyset$, then $a_i$ dominates $S$.
\medskip

The concept of meta-stability is a weakening of the stability one, as the following assertion shows.\medskip

\noindent\textbf{Proposition 4.1.} \emph{A stable contract  network is meta-stable.}\medskip

\begin{proof} Suppose that a contract $d$ dominates an individually rational network $S$. We show that $d$ blocks $S$. To see this, let $i$ be an arbitrary participant of $d$. We explain that $d\notin S$ and $d\in f_i(S(i)\cup d)$. This is obvious if $S(i)$ empty. So assume that $S(i)\ne\emptyset$ and let $s_i$ be a contract from $S(i)$ that is not selected from the pair $\{d,s_i\}$. Then, moreover, $s_i$ is not selected from the set $S(i)\cup d$, i.e., $s_i\notin f_i(S(i)\cup d)$. This is possible only if $d\in f_i(S(i)\cup d)$, for otherwise $f_i(S(i)\cup d)=f_i(S(i))$ (by Outcast property, see Appendix C). By \textbf{S0}, the last set  is equal to  $S(i)$ and contains $s_i$. So $s_i\in f_i(S(i)\cup d)$, contrary to the assumption.

Also $d\notin S$. Thus, $d$ blocks $S$, contrary to condition \textbf{S*}.
 \end{proof}

We see that the difference of meta-stability from stability consists in rejection of  individual rationality. That is, we admit the possibility that $S(i)$ consists not only of all `best' contracts, but may also include some `worse' ones. Why does the agent not refuse `bad' contracts? One justification (though not convincing enough) is as follows. It may happen that a contract $s$, which agent $i$ would like to get rid of, is unique for some other agent $j$, i.e.,  $S(j)=\{s\}$. If $s$ is removed from $S$, then the agent $j$ remains without contracts at all. This forces him to agree to some other contract $c$, which earlier was regarded by him as not good enough. But then another participant $k$ of the contract $c$ may refuse an earlier concluded contract, which now becomes uninteresting for $k$ due to the appearance of $c$. In a word, a cascade of contract renegotiations may begin with an unpredictable outcome for the `initiator' $i$.

Let us illustrate this situation with a simple example. There are three agents, say, 1,2,3, and six contracts, of which three are autarkic (of the form $\{i,i)\}$, where $i=1,2,3$) and the other three are 2-element ones, namely, $\{i,j\}$ for $i\ne j$. The preferences of agent $i$ are arranged as follows: $\{i,i\}<\{i,i-1\}<\{i,i+1\}$ (letting 0=3 and 4=1). That is, agent $i$ considers as most preferable for him to have the contract with agent $i+1$ (taken modulo~3), and as least preferable to stay alone. In this example there is no stable contract system. But there are meta-stable ones. One of them consists of two contracts $\{1,2\}$ and $\{2,3\}$. (It is easy to see that the only contract that could pretend to dominate is $\{3,1\}$, but it is not better for agent 1 than the contract \{1,2\}.) Here agent 2 enters into two contracts, and the contract with agent 1 is worse for him than the one with agent 3. Imagine that he breaks the contract with agent 1. Then agent 1, who does not want to remain alone, turns to agent 3. Agent 3 is glad to accept his offer, refusing the old contract with agent 2. As a result, agent 2 looses both contracts and is forced to get his autarkic contract, which is worse than what he had before.

Another, perhaps more interesting justification for the rejection of individual rationality, is as follows. Let us regard an agent not as an `individual', but rather as a `collective' of several subagents which are equivalent in terms of preferences. Some of them succeed to conclude more profitable contracts, while the others become less profitable ones. As in illustration, in the above example one can split agent 2 into two subagents $2_1$ and $2_3$, and consider the concluded contracts of the form $\{1,2_1\}$ and $\{2_3,3\}$.


\section{Existence of meta-stable contract systems}

We have seen that the meta-stability represents a weakened concept of stability. The main advantage of this concept is that a meta-stable contract system `always' exists. Here, by saying  `always', we assume that all choice functions $f_i$ that we deal with are non-empty-valued Plott ones.\medskip

\noindent\textbf{Theorem 5.1.} \emph{For choice functions as above, a meta-stable contract system does exist.}\medskip

The proof consists of two parts. In the first one, we prove the theorem in the special case when the preferences of all agents are given by linear orders. To do this, we use a certain general `theorem on compromise' whose formulation and proof are left to Appendix~B. In the second part, we reduce a general case to the linear one.\medskip

\emph{Linear case.} Here we assume that the preferences of each agent $i$ are given by a linear order $\le _i$ on $C(i)\subseteq C$. Add for each agent $i$ a `dummy' autarkic contract $a_i$, which is worse than the others elements of $C(i)$. Below we shall show that in this `extended' system with $\widehat{C}=C\cup\{\{a_i\}, i\in I\}$ there exists a meta-stable set $\widehat{S}$ with all $\widehat{S}(i)$ nonempty.

Relying on this, we assert that $S=\widehat{S}\cap C$ \emph{is meta-stable in the initial problem}.

Suppose, for a contradiction, that there is a contract $d\in C$ dominating $S$ and let $i$ be an arbitrary participant of $d$. We claim that $d$ dominates $\widehat{S}(i)$ for  $i$, thus  contradicting the meta-stability of $\widehat{S}$. Indeed, if $S(i)=\emptyset$, then $\widehat{S}(i)=\{a_i\}$, and since $a_i<_i d$, the contract $d$ dominates $\widehat{S}(i)$ for $i$. And if $S(i)\ne\emptyset$, then the element in $S(i)$ which is worse than $d$ will be worse than $d$ in $\widehat{S}(i)$ as well. So in all cases $d$ dominates $\widehat{S}(i)$. Since this is true for any participant of $d$, we obtain that $d$ dominates $\widehat{S}$, contrary to the supposition.

In light of the proof, we can additionally assume that each agent has an autarkic contract.

Let $u_i:C(i) \to \mathbb R$ denote a utility function on $C(i)$ representing the order $\le_i$. Using this, each contract $c\in C$ can be considered as a `partially defined' function (keeping the same notation $c$) on $I$, given by the rule
                                                                    $$
                                   c(i):=u_i(c).
                                                                      $$
The domain of this function is the set $P(c)\subseteq I$ of participants of $c$. By the theorem on compromise that we discuss in Appendix~B, there exists a `compromise' function $x$, already defined on the whole $I$, which possesses two properties:

1) any $c\in C$ is smaller than or equal to $x$ at some point $i$ in the domain $P(c)$ of $c$;

2) for any $i\in I$, there exists $c\in C(i)$ which is not smaller (strictly) than $x$ within the whole domain $P(c)$ of $c$.

When such a function $x$ is available, we define the set $S$ as
                                                         $$
           S:=\{c\in C \,\colon\, x\le c\;\; \mbox{within the domain of $c$}\}.
        $$
Note that, due to property 2), the sets $S(i)$ are nonempty for all $i\in I$.\medskip

\noindent\textbf{Claim.} \emph{The set of contracts $S$ is meta-stable.}\medskip

\begin{proof}
Suppose that some contract $d$ dominates  $S$. Since all sets $S(i)$ are nonempty, this means that for any $i\in P(d)$, there is a contract $s_i\in S(i)$ (regarded as a function) such that $s_i(i)<d(i)$. Since $s_i\in S$, we have $x(i)\le s_i(i)$; then $x(i)<d(i)$ for any $i$ from the domain of $d$. But this contradicts property 1).
 \end{proof}

\emph{Linearization.} Let $f$ be a choice function on a set $X$. Let us say that a linear order  $\le$ on $X$ \emph{respects} $f$ if for any nonempty $A\subseteq X$, the maximal (relative to $\le$) element in $A$ belongs to $f(A)$. If $f$ is a non-empty-valued Plott function, then there are `a lot of' linear orders respecting $f$. Here `a lot of' means that for any $a\in f(A)$, there is a  linear order $\le$ respecting $f$ such that $a$ is the largest element in $A$. For more information, see Appendix~C.

Recall that an element $d$ dominates a nonempty set $A$ (relative to a CF $f$) if $a\notin f(a,d)$ for some $a\in A$. The following assertion is easy.\medskip

\noindent\textbf{Lemma 5.2.} \emph{Let $f$ be a non-empty-valued Plott function on $X$, and let a linear order $\le$ respect $f$. If $d$ dominates $A\subseteq X$ relative to $f$, then $d$ dominates $A$ relative to $\le$.}\medskip

\begin{proof} If $d$ dominates $A$ relative to $f$, then there is $a\in A$ such that $a\notin f(a,d)$. But then $a < d$ for any order $\le$ respecting $f$.
\end{proof}

Let us go back to our problem with agents $I$ and Plott functions $f_i$. Consider a new problem (an equipment) in which each Plott function $f_i$ is replaced by a linear order $\le _i$ which respects $f_i$. In the `linearized' problem $(I,C,\{\le _i, i\in I\})$ there is a meta-stable contract system $S$. We assert that $S$ is meta-stable in the initial problem as well. To see this, we have to check that there are no dominating contracts. But if $d$ dominates $S$ relative to $f_i$ (where $i$ is in $P(d)$), then, by Lemma 5.2, $d$ dominates $S$ relative to $\le _i$, which contradicts to meta-stability of $S$ in the problem with linear orders.

This completes the proof of Theorem 5.1 (assuming validity of the theorem on compromise).


\section{Minimal meta-stable networks}

In this section we assume that each agent has an autarkic contract.

The concept of meta-stable networks is not rigid enough. Let $S$ be a meta-stable contract system, and $T\subseteq S$ some subsystem in it. If $T(i)$ is not empty for every agent $i$, then  the system $T$  is also meta-stable. Note that stable networks do not allow a similar possibility; namely, if $S$ is stable, then  any proper subsystem in it is already unstable. This justifies the following notion. \medskip

\noindent{\bf Definition.}  A meta-stable system is called {\em  minimal} if it is minimal by inclusion, that is, any proper subsystem in it is not meta-stable.\medskip

There is a simple criterion of the minimality. Let $S$ be a system of contracts. Let us say that agent $i$ is \emph{monogamous}\footnote{The term monogamous is appropriate if `$\gamma\alpha\mu o\zeta$' (marriage) is understood as a contract.} if $S(i)$ consists of a single contract. This situation is typical in classical marriage or roommate problems.

\begin{prop} \label{prop4}
A meta-stable system $S$ is minimal if and only if any contract $s$ of $S$ contains a monogamous participant.
\end{prop}

\begin{proof} It follows from the obvious fact that if $s$ is a contract without monogamous participants, then the system $S-\{s\}$ is also meta-stable.
 \end{proof}

Similar reasonings show that any meta-stable contract system is a union of minimal meta-stable ones. Therefore, in principle, we can restrict ourselves by studying minimal meta-stable networks. Especially since such systems provide the largest guaranteed utility.\medskip

\emph{Linearization.} Due to the reduction theorem 3.1, we may assume that the preferences of agents are given by weak orders. Let $\le _i$ be a weak order of agent $i$ on the set $C(i)$, and let $\preceq _i$ be a linear order extending $\le _i$ on the same set. (Then $c <_i c'$ implies $c\prec _ic'$, that is, strict preferences preserve while equivalences are eliminated. In terms of utilities, the original utility functions are slightly perturbed.) So, preserving the hypergraph $(I,C)$, we strengthen the initial preferences of agents to get linear orders $\preceq$.

 \begin{prop} \label{prop5}
1) If $S$ is a meta-stable network with respect to linear orders  $\preceq$, then $S$ is meta-stable for the original weak orders $\le$ as well.

2) Conversely, if  $S$ is  a minimal meta-stable network for $\le$, then there are corresponding linear extensions $\preceq$ of $\le$ such that $S$ is meta-stable with respect to $\preceq $.
\end{prop}
 \begin{proof}
Assertion 1) is easy. We only need to check that $S$ is non-dominated for the weak orders $\le$. Suppose this is not so, that is, there is a contract $c$ such that $u_i(c)>u_i(S)$ for any participant $i$ of $c$. But then $\tilde u_i(c)>\tilde u_i(S)$ for all $i\in P(c)$, which contradicts the meta-stability of $S$ with respect to $\preceq$ (where $\tilde u$ stands for the utilities for $\preceq$).

To see 2), let $S$ be a meta-stable system for the weak orders $\le_i$. The `splitting of ties' of non-marginal contracts is not important in essence, so we can focus on marginal ties. Fix a participant $i$ and denote by $M(i)$ the set of contracts $c\in C(i)$ such that $u_i(c)=u_i(S)$. Such contracts are divided into two groups. The group of those belonging to $S$ is denoted by $M_+(i)$, and the rest by $M_-(i)$. The first group is certainly nonempty. Choose some contract $s_i$ in it, remain its utility unchanged, and slightly increase the utilities $\tilde u_i$ of the other members of  $M_+(i)$. Also slightly decrease the utilities of contracts in $M_-(i)$.

Doing so for all agents, we eventually obtain a system with linear orders $\preceq _i$ extending the original weak orders $\le_i$. We assert that this system  is meta-stable.

Indeed, let $c$ be an arbitrary contract, and suppose that for its all participants $i$, there holds $\tilde u_i(c)>\tilde u_i(S)$. Then $u_i(c)\ge u_i(S)$. Such inequalities cannot be strict for all $i$, in view of the meta-stability of $S$. So there is $i$ for which the equality $u_i(c)=u_i(S)$ is fulfilled; then $c\in M(i)$. If $c\in M_-(i)$, then its perturbed utility $\tilde u_i$ is slightly less than the utility of marginal contract $s_i$, contrary to the supposition $\tilde u_i(c)>\tilde u_i(S)$. Hence $c\in S$. Now, since $\tilde u_i(c)>\tilde u_i(S)$, the contract $c$ is not unique in $S(i)$, for all $i\in P(c)$. This contradicts the minimality of $S$.
 \end{proof}

We obtain that when dealing with minimal meta-stable networks, one may assume, w.l.o.g., that the preferences of all participants are given by linear orders. In other words, all contracts for any agent $i$ are comparable and non-equivalent. In this situation, we have a closer relationship between the stability and minimal meta-stability, as follows.
 \begin{prop} \label{prop6}
Suppose that the preferences of all agents are given by linear orders. If $S$ is a stable contract network, then $S$ is a minimal meta-stable one.
  \end{prop}
\begin{proof} The meta-stability has already been established. It remains to check that any contract $s\in S$ is unique for some of its participants $i$. (Equivalently, any contract $s\in S$ has a monogamous participant.) This follows from the fact that (in the case of linear preferences) each set $S(i)$ consists of a single contract.      \end{proof}

Note that when the preferences are linear and the original network $C$ is binary (viz. $(I,C)$ is a  graph), a minimal meta-stable network is split into connected components that are either isolated vertices or have the shape of `dandelions':
 \medskip

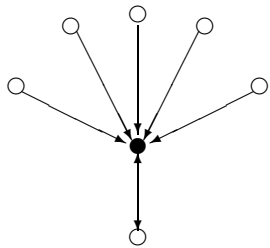
\unitlength=.8mm
\special{em:linewidth 0.5pt}
\linethickness{0.5pt}
\begin{picture}(92,45)(-20,0)
\put(70.00,20.00){\circle*{2.83}}
\put(70.00,5.00){\circle{2.83}}
\put(50.00,30.00){\circle{2.83}}
\put(59.00,40.00){\circle{2.83}}
\put(70.00,42.00){\circle{2.83}}
\put(81.00,40.00){\circle{2.83}}
\put(90.00,30.00){\circle{2.83}}
\put(70.00,19.00){\vector(0,-1){13.00}}
\put(70.00,6.00){\vector(0,1){13.00}}
\put(51,29.00){\vector(2,-1){17.00}}
\put(60.00,39.00){\vector(1,-2){9.00}}
\put(70.00,40.00){\vector(0,-1){18.00}}
\put(80.00,39.00){\vector(-1,-2){9}}
\put(89.00,29.00){\vector(-2,-1){17.00}}
\end{picture}

\noindent Here there is a central agent (the dark vertex), and several monogamous agents associated with it. The marginal contract of the central agent is represented by a two-sided arrow. In a particular case, this structure degenerates into a single binary contract.

\section*{Appendixes}

\appendix

\section{Proofs of Propositions~\ref{prop1} and~\ref{prop2}}

To prove these propositions, we use the following lemma, where $S(1)=\pi (\widetilde S(1))$, $S(2)=\pi(\widetilde S(2))$, and $S(j)=\pi (\widetilde S(j))$ for $j\ne 0,1,2$.

\begin{lemma} \label{lemma1}
\begin{itemize}
\item[\rm(a)] $f_j(S(j))=S(j)$ for each agent $j$ different from 0.
\item[\rm(b)] $f_1(S(1))=f_1(S(0))$, and similarly $f_2(S(2))=f_2(S(0))$.
 \end{itemize}
\end{lemma}

\begin{proof}
(a) Let $s\in S(j)$. If all participants of the contract $s$ are different from 0, then $s=\pi (\widetilde s)$ for a single (actually, equal to $s$) contract $\widetilde s$ from $\widetilde S$. By condition~\textbf{S0}, $ \widetilde s\in\widetilde f_j(\widetilde S(j))$; then $s\in f_j(S(j))$. Therefore, we may assume that 0 is one of the participants of $s$. The contract $s$ is a projection of some $\widetilde s$ from $\widetilde S$, and 1 or 2 is a participant of $\widetilde s$. Let for definiteness $\widetilde s=s_1\in\widetilde S(1)$; then $s_1\in\widetilde S(j)$. By condition~\textbf{S0}, for $\widetilde S$ we have the equality $\widetilde f_j(\widetilde S(j))=\widetilde S(j)$. Thus, $s_1\in\widetilde f_j(\widetilde S(j))$. By definition of $\widetilde f_j$, this means that  $s=\pi(s_1)$ belongs to $f_j(\pi(S(j)))=f_j(S(j))$.

(b) Recall that $S(0)=S(1)\cup S(2)$. It is enough to show that $f_1(S(0))\subseteq S(1)$, since then Outcast property (defined in Appendix~C) gives the desired equality.

Suppose, for a contradiction, that some contract $s$ of $f_1(S(0))$ does not belong to $S(1)$. Since $s$ is selected from the larger set $S(0)$, it is also selected from the smaller set $S(1)\cup s$, yielding $s\in f_1(S(1)\cup s)$. Then $s_1\in \widetilde f_1(\widetilde S(1)\cup s_1)$. This shows that the contract $s_1$ is not autarkic. Since in this case $s_1\in \widetilde S(1)$ and $s\in S(1)$, contrary to the supposition.

Now let $j$ be another participant of the contract $s$ (or $s_1$). We assert that $s_1\in\widetilde f_j(\widetilde S(j)\cup s_1)$. To show this (see the definition of $\widetilde f_j$), we have to make sure that $\pi (s_1)=s$ belongs to $f_j(\pi (\widetilde S(j)\cup s_1))=f_j(S(j)\cup s)$. Note that $s_2\in \widetilde S(2)$; then $s_2\in \widetilde S(j)$ and $s\in S(j)$. So $S(j)\cup s=S(j)$, and $f_j(S(j)\cup s)$ is equal to $S(j)$ and contains $s$.

Thus, $s_1$ belongs to both $\widetilde f_1(\widetilde S(1)\cup s_1)$ and $\widetilde f_j(\widetilde S(j)\cup s_1)$ for any participant $j\ne 1$ of $s_1$. By condition~\textbf{S*}, we obtain $s_1\in \widetilde S(1)$ and $s\in S(1)$, yielding a contradiction.
 \end{proof}

\noindent\textbf{ Proof of Proposition~\ref{prop1}.} One has to verify properties \textbf{S0} and \textbf{S*} for the system $S$.

We first check \textbf{S0} for agent 0, that is, $f_0(S(0))=S(0)$. Let $s\in S(0)$; one may assume that $s\in S(1)$. Since $S(1)=f_1(S(1))$ (according to~\textbf{S0} for $\widetilde S$ at the vertex 1), $s$ belongs to $f_1(S(1))$, which is equal to $f_1(S(0))$ (by Lemma~\ref{lemma1}(b)), and therefore $s$ belongs to $f_0(S(0))$. For other agents $j$, the needed equality is established in Lemma~\ref{lemma1}(a).
\smallskip

Next we check~\textbf{S*}. Suppose that there is a blocking contract $b$ for $S$. That is, $b\notin S$, but $b\in f_i(S(i)\cup b)$ for any $i\in P(b)$. If $0\notin P(b)$, then $b$ also blocks $\widetilde S$. Therefore, we may assume that $0\in P(b)$.

In this case,  $b\in f_0(S(0)\cup b)$. This means that $b$ lies either in $f_1(S(0)\cup b)$ or in      $f_2(S(0)\cup b)$. Let $b\in f_1(S(0)\cup b)$. Due to Heredity property of CF $f_1$, we have  $b\in f_1(S(1)\cup b)$. Then $b_1\in \widetilde f_1(\widetilde S(1)\cup b)$, where $b_1=(b,1)$.

Now let us examine the inclusion $b\in f_j(S(j)\cup b)$, where $j\in P(b)-\{0\}$. By the definition of $\widetilde f_j$, we have $b_1\in\widetilde f_j(\widetilde S(j)\cup b_1)$ since $\pi(\widetilde S(j)\cup b_1)=S(j)\cup b$.

Finally, $b_1$ does not belong to $\widetilde S$ since $b\notin S$. Then the contract $b_1$ is blocking for $\widetilde S$, contrary to the stability of $\widetilde S$. \hfill$\Box$\medskip

\noindent\textbf{Proof of Proposition~\ref{prop2}.}
Recall how the `covering' $\widetilde S$ is arranged. If $s$ belongs to $S$ and does not contain 0 as a participant, then $s$ is lifted in $\widetilde C$ in a natural way and is included in $\widetilde S$. As to $\widetilde S(1)$ and $\widetilde S(2)$, they are defined as $\widetilde S(1):=f_1(S(0))$ and $\widetilde S(2):=f_2(S(0))$ (where we identify $\widetilde C(1)$ and $\widetilde C(2)$ with $C(0)$). Since $f_1(S(0))\cup f_2(S(0))=f_0(S(0)) =S(0)$ (by condition \textbf{S0}), we have
                                                                  $$
\pi (\widetilde S(1)\cup \widetilde S(2))=S(0).
               $$
Similar equalities hold for vertices $j$ different from 0. We need two additional lemmas.

 \begin{lemma} \label{lemma2}
Let a vertex $j$ be different from $0$. Then $\pi (\widetilde
S(j))=S(j)$.
 \end{lemma}
 \begin{proof}
Let $s\in S(j)$. We have to show that $s$ appears from $\widetilde S(j)$. This is immediate from the construction if  agent 0 does not participate in the contract $s$. So assume that $0\in P(s)$. Then $s\in S(0)$, and by condition \textbf{S0}, the contract $s$ is selected from $S(0)$ either by CF $f_1$ or by  CF $f_2$. Assume that $s\in f_1(S(0))$; then $s_1\in\widetilde S(1)$. Therefore, $s=\pi(s_1)$ belongs to $\pi(\widetilde S)$, whence $s\in\pi(\widetilde S(j))$.

Conversely, let $s=\pi (\widetilde s)$ for $\widetilde s\in \widetilde S(j)$, and let for definiteness $1\in P(\widetilde s)$. Since $\widetilde s\in\widetilde S(1)$, we have $s\in f_1(S(0))\subseteq S(0)$, which means that $s$ belongs to $S(j)$.
 \end{proof}
\begin{corollary} \label{coro1}
$\pi (\widetilde S)=S$.
 \end{corollary}
 \begin{lemma} \label{lemma3}
Each $c\in C(0)$ satisfies $\widetilde f_1(\widetilde S(1)\cup c_1)=f_1(S(0)\cup c)$.
 \end{lemma}
 \begin{proof} $\widetilde S(1)=f_1(S(0))$. Therefore $\widetilde f_1(\widetilde S(1)\cup  c_1)=f_1(f_1(S(0)\cup  c)=f_1(S(0)\cup  c)$.
  \end{proof}

Now we are ready to finish the proof of Proposition~\ref{prop2}. One has to verify properties \textbf{S0} and \textbf{S*} for $\widetilde S$.
 \medskip

\noindent\emph{Verification of}~\textbf{S0.} We have to  show that $\widetilde f_i(\widetilde S(i))=\widetilde S(i)$ for any vertex $i$ of the hypergraph $\widetilde G$.

For vertex 1, we have $\widetilde S(1)=f_1(S(0))=f_1(f_1(S(0)))=f_1(S(1))=\widetilde      f_1(\widetilde S(1))$. Similarly for vertex 2.

Now let $j$ be different from 1 and 2. Let $\widetilde s\in\widetilde S(j)$. We have to show that $\widetilde s$ is selected (by CF $\widetilde f_j$) from $\widetilde S(j)$. This is obvious if none of 1 and 2 occurs among the participants of  $\widetilde s$. Assume that agent 1 participates in $\widetilde s$, that is, $\widetilde s=s_1$ for some edge $s$ of $f_1(S(0))\subseteq S(0)$. But then $s\in S(j)=f_j(S(j))$, by~\textbf{S*}. And since $s=\pi(s_1)$ is selected (by CF $f_j$) from $S(j)=\pi (\widetilde S(j))$ (see Lemma~\ref{lemma2}), we have $s_1\in \widetilde f_j(\widetilde S(j))$ (by the definition of $\widetilde{f}_j$).
 \medskip

\noindent\emph{Verification of}~\textbf{S*}. We show that there are no blocking contracts for $\widetilde S$. Suppose, for a contradiction, that  such a contract $\widetilde b$ exists. If neither 1 nor 2 occurs among the participants of $\widetilde b$, then its projection $b=\pi(\widetilde b)$ blocks $S$, which contradicts the stability of $S$. Therefore, we may assume that agent 1, say, participates in  $\widetilde b$.

Assume that 1 is the unique participant of $\widetilde b$ (that is, $\widetilde b$ is autarkic). Then $\widetilde b$ does not belong to $\widetilde S(1)$, and $b=\pi(\widetilde b)$ does not belongs to $f_1(S(0))$. If $b\in S(0)$, then $b\notin f_1(S(0)\cup b)=f_1(S(0))$. But due to Lemma~\ref{lemma3}, $f_1(S(0)\cup b)=\widetilde f_1(\widetilde S(1)\cup b_1)$. So $\widetilde b=b_1$ is not blocking, contrary to the supposition. Therefore, we may assume that $b\notin S(0)$. Since $S$ is stable, $b\notin f_0(S(0)\cup b)$ and, moreover, $b$ does not belong to  $f_1(S(0)\cup b)$. This contradicts the stability of $S$.

Now consider a participant $j$ in $P(b_1)$ different from 1. Since $b_1$ blocks $\widetilde S$, we have

           1) $b_1\notin \widetilde S(1)$;

           2) $b_1\in \widetilde f_1(\widetilde S(1)\cup b_1)$, and

           3) $b_1\in \widetilde f_j(\widetilde S(j)\cup b_1)$.

The first relation can be rewritten as $b\notin f_1(S(0))$.

The second relation can be rewritten as $b\in f_1(S(0)\cup b)$ (in view of Lemma~\ref{lemma3}). This together with the first relation implies that $b\notin S(0)$.

The third relation (with the definition of $\widetilde f_j$) gives   $b\in f_j(\pi (\widetilde S(j)\cup b))=f_j(S(j)\cup b)$ (by Lemma~\ref{lemma2}).

As a consequence, we obtain that $b$ blocks $S$, contrary to the stability of $S$. This completes the proof of Proposition 2. \hfill$\Box\Box$


 \section{Theorem on compromise}

As in Section~5, we will think of each contract $c$ as a partially defined real-valued function on the set of agents $I$. The definition domain $Dom(c)$ of this function coincides with $P(c)$, the set of participants of $c$, and the value at $i\in Dom(c)$ is defined to be $u_i(c)$. Thus, $C$ is a finite set of partially defined functions on $I$. We assume that for each $i$, the set $C$ contains an `autarkic' function defined only at the point $i$. \medskip

\noindent\textbf{Definition.} A function $x:I\to\mathbb R$ is called {\em compromise} for $C$ if the following two properties are satisfied:

1) no function $c$ in $C$ can be strictly greater than $x$ within its domain; in other words, for any $c\in C$, there exists $i\in Dom(c)$ such that $c(i)\le x(i)$;

2) for any participant $i\in I$, there exists $c\in C$ such that $i\in Dom(c)$ and $x\le c$ (within the domain of $c$).\medskip

The first property is something like a coalition rationality: the coalition $Dom(c)$ refuses the `distribution' $x$ if $c$ gives strictly more than $x$ to every participant of the coalition. In particular, $x$ is no worse than any autarkic contract. The second condition says that $x$ cannot be too large: the `payment' to any agent $i$ must be `justified' by its participation in some `good' contract (which gives at least $x(i)$ to all participants $i$ of the contract).
\begin{theorem} \label{theorem2p}
A compromise function $x$ does exist.
  \end{theorem}

To show this, one could appeal to Scarf lemma; yet we prefer to give a direct and concise proof inspired by~\cite{Dan99}. We will construct a correspondence $F$ (of the form $x\mapsto F(x)$) whose fixed points coincide with the compromises. The existence of fixed points will follow from Kakutani's theorem.\medskip

\noindent\emph{Construction of the correspondence} $F$. Take a `big cube' $X=[-N,N]^I$ in the space $\mathbb R^I$ (where $N$ is large compared with the maximal value of functions in $C$) and construct a convex-valued correspondence $F:X\Rightarrow X$. To do this, one needs to define the `image' $F(x)$ of any point $x\in X$. This set $F(x)$ is constructed as a parallelepiped of the form $\times (F_i(x)\colon i\in I)$, where $F_i(x)$ is a closed segment in $[-N,N]$. Moreover, $F_i(x)$ is assigned as either the whole segment $[-N,N]$ or one of its ends.

To do this, fix $i$ and consider the set $C(i)=\{c\in C\colon i\in Dom(c)\}$. We define
$$
F_i(x)=\left\{\begin{array}{ll}
\{N\}, & \text{if there exists $c\in C(i)$ such that $x<c$ on the domain of $c$.} \\
\{-N\}, & \text{if every function $c$ in $C(i)$ is strictly less than $x$ at some point $j$.} \\
\lbrack -N,N \rbrack & \text{otherwise.}
 \end{array} \right.
 $$

\begin{lemma} \label{lemma}
$x\in F(x)$ if and only if $x$ is a  compromise function for $C$.
 \end{lemma}
\begin{proof}
Let us check property 1) in the compromise definition. Suppose this is not valid, namely, there is a function $c$ strictly greater than $x$ (within $Dom(c)$). Then for any $i\in Dom(c)$, we have $F_i(x)=\{N\}$ and  $x(i)=N$. But then $c(i)=N$, which contradicts the definition of $N$.

Next we check property 2). Suppose, for a contradiction, that there is $i$ such that every function $c$ in $C(i)$ is somewhere less than $x$ (that is, $c(j)<x(j)$ for some $j\in Dom(c)$). Then $F_i(x)=\{-N\}$ and $x(i)=-N$. In particular, an autarkic function $a_i$ is defined only at the set $\{i\}$ and, therefore, it is less that $x$, $a_i(i)<x(i)=-N$. This again contradicts the definition of $N$.
 \end{proof}

The existence of fixed points of $F$ follows from Kakutani's theorem. Indeed, the images of $F$ are convex and nonempty. Therefore, it suffices to show that the graphic of each correspondence $F_i$ is closed. This is a consequence of the fact that if $F_i(x)$ is $\{N\}$ or $\{-N\}$, then so is for the points $x'$ close to $x$ as well.

The converse assertion is trivial. \hfill$\Box$


\section{Plott choice functions}

In this section, we assume (for simplicity) that $X$ is a finite set. Recall that a CF on $X$ is a mapping $f:2^X\to 2^X$ such that $f(A)\subseteq A$ for any `menu' $A\subseteq X$. Such a CF is called a {\em Plott function} if the following equality holds for any menus $A$ and $B$:
      $$
                                f(A\cup B)=f(f(A)\cup B).
      $$
This immediately gives $f(A\cup B)=f(f(A)\cup f(B))$, as well as $f(f(A))=f(A)$.

Let us fix some Plott CF $f$. A subset $N\subseteq X$ is called {\em null} (or insignificant) if $f(N)=\emptyset$. It can be seen that the union $N^*$ of all null sets is a null set as well (the largest null set). Adding any null set to a menu does not change the choice. So, by removing $N^*$ from $X$, we may assume that $\emptyset$ is the only null set, that is, assume that $f$ is a `non-empty-valued' CF.

           Plott functions have two characteristic properties.\medskip

\noindent\emph{Heredity} (or substitutability): if $A\subseteq B$, then $f(B)\cap A\subseteq f(A)$. In other words, if $a\in A$ is chosen in a larger set $B$, then $a$ is chosen in $A$ as well.\medskip

\noindent\textbf{Corollary.} \emph{If $A\subseteq f(B)$, then $f(A)=A$.}\medskip

\noindent\emph{Outcast} (or independence from rejected alternatives, IRA): if $f(A)\subseteq B\subseteq A$, then $f(A)=f(B)$.\medskip

Conversely, it can be shown that holding Heredity and Outcast properties implies that the CF is Plottian.\medskip

It is easy to check that the union of (two or more) Plott functions is again Plott function. Aizerman and Malishevski~\cite{AM} showed that any (non-empty-valued) Plott function can be represented as the union of several linear CFs.

Let $f$ be a CF on a finite set $X$. Let's say that a linear order $\le$ on $X$ \emph{respects} $f$ if for any (nonempty) $A\subseteq X$ the largest (relative to $\le$) an element of $A$ belongs to $f(A)$. If such an order exists, $f$ is non-empty-valued. In the case of a non-empty-valued Plott       function, there are a lot of linear orders respecting $f$. This is a reformulation of the Aizerman-Malishevski theorem.

The next assertion has been encountered earlier. Let $\pi :X\to Y$ be a mapping of sets, and $g$ a CF on $Y$. Define the CF $f=\pi^*(g)$ on $X$ by the following formula:
                                                   $$
               f(A)=A\cap \pi ^{-1}(g(\pi (A)))\quad \mbox{for $A\subseteq X$}.
                                                     $$
In other words, $a$ is chosen from $A$ if $\pi(a)$ is chosen from $\pi(A)$. In particular, $\pi(f(A))=g(\pi (A))$.

 \begin{prop} \label{prop5}
The CF $f$ is Plottian if $g$ is Plottian.
  \end{prop}
\begin{proof} Let us check that $f$ satisfies Heredity and Outcast.

To see Heredity, let $A\subseteq B$, $a\in A$ and $a\in f(B)$. Then $\pi(a)\in g(\pi (B))$. Since $\pi (A)\subseteq \pi (B)$, by Heredity for $g$, we get $\pi(a)\in g(\pi (A))$, that is, $a\in f(A)$.

To see Outcast, it suffices to show that if $f(B)\subseteq A\subseteq B$, then $f(A)\subseteq f(B)$ (the converse inclusion follows from Heredity). Let $a\in f(A)$; then $\pi(a)\in g(\pi (A))$. Applying $\pi $ to the inclusions $f(B)\subseteq A\subseteq B$, we get $g(\pi (B))=\pi (f(B))\subseteq \pi (A)\subseteq \pi (B)$. From Outcast for $g$, we get $g(\pi (A))\subseteq g(\pi (B))$. Then $\pi(a)\in g(\pi(B) )$, implying $a\in f(B)$.
 \end{proof}

Obviously, for $\pi :X\to Y$ as above, if CF $g$ is given by a weak order $\le _Y$, then $f=\pi^\ast(g)$ is given by the weak order $\le _X=\pi^{-1}(\le _Y)$.
\medskip

\noindent\textbf{Acknowledgments.} We thank the anonymous reviewers for meticulously reading the original version of this paper and many useful suggestions.

      \end{document}